\documentclass[12pt]{amsart}

\usepackage{amsmath}
\usepackage{amsfonts}
\usepackage{amssymb}
\usepackage{latexsym}
\usepackage{graphicx}
\usepackage{enumerate}
\usepackage{multirow}
\usepackage{subfigure}

\numberwithin{equation}{section}

\newtheorem{theorem}{Theorem}[section]

\newtheorem{lemma}[theorem]{Lemma}
\newtheorem{cor}[theorem]{Corollary}
\newtheorem{remark}{Remark}

\usepackage{mathtools}

\makeatletter
\DeclareRobustCommand\widecheck[1]{{\mathpalette\@widecheck{#1}}}
\def\@widecheck#1#2{%
    \setbox\z@\hbox{\m@th$#1#2$}%
    \setbox\tw@\hbox{\m@th$#1%
       \widehat{%
          \vrule\@width\z@\@height\ht\z@
          \vrule\@height\z@\@width\wd\z@}$}%
    \dp\tw@-\ht\z@
    \@tempdima\ht\z@ \advance\@tempdima2\ht\tw@ \divide\@tempdima\thr@@
    \setbox\tw@\hbox{%
       \raise\@tempdima\hbox{\scalebox{1}[-1]{\lower\@tempdima\box
\tw@}}}%
    {\ooalign{\box\tw@ \cr \box\z@}}}
\makeatother

\newcommand{\om}{\omega}

\DeclareMathOperator{\sech}{sech}

\newcommand{\ds}{\displaystyle}

\newcommand{\be}{\begin{equation}}
\newcommand{\ee}{\end{equation}}
\newcommand{\bes}{\begin{equation*}}
\newcommand{\ees}{\end{equation*}}

\newcommand{\R}{{\bf{R}}}
\newcommand{\C}{{\bf{C}}}

\renewcommand{\L}{{\mathcal{L}}}

\renewcommand{\O}{{\mathcal{O}}}
\newcommand{\Fo}{{\mathfrak{F}}}

\renewcommand{\hat}{\widehat}

\newcommand{\bunderbrace}[2]{%
  \begin{array}[t]{@{}c@{}}
  \underbrace{#1}\\
  #2
  \end{array}
}

\title{A simple model of radiating solitary waves}
\author{J. Douglas Wright}

\begin{document}
\maketitle
\begin{abstract}
To understand an oft-observed but poorly understood phenomenon in which a solitary wave in a dispersive equation slowly deteriorates due a persistent emission of radiation ({\it i.e.} a ``radiating solitary wave''), we propose a bare-bones model which captures many essential features and which we are capable of analyzing completely by way of the Laplace transform.
We find that  wave amplitude decreases at an exponential rate but with a decay constant that is (in many cases) small beyond
all orders of the frequency.
\end{abstract}


\section{Introduction}

In the articles \cite{okada, tabata, GSWW}, simulations of solitary waves in spatially heterogeneous variants of the Fermi-Pasta-Ulam-Tsingou (FPUT) and Toda lattices demonstrated that such waves do not propagate without change of form but instead continuously emit a small trailing ripple. The systems conserve energy and consequently the solitary waves experience a commensurate and extremely slow attenuation in amplitude. See Figure~\ref{co-ops} for a representative depiction.\begin{figure}
\centering
    \includegraphics[width=.95\textwidth]{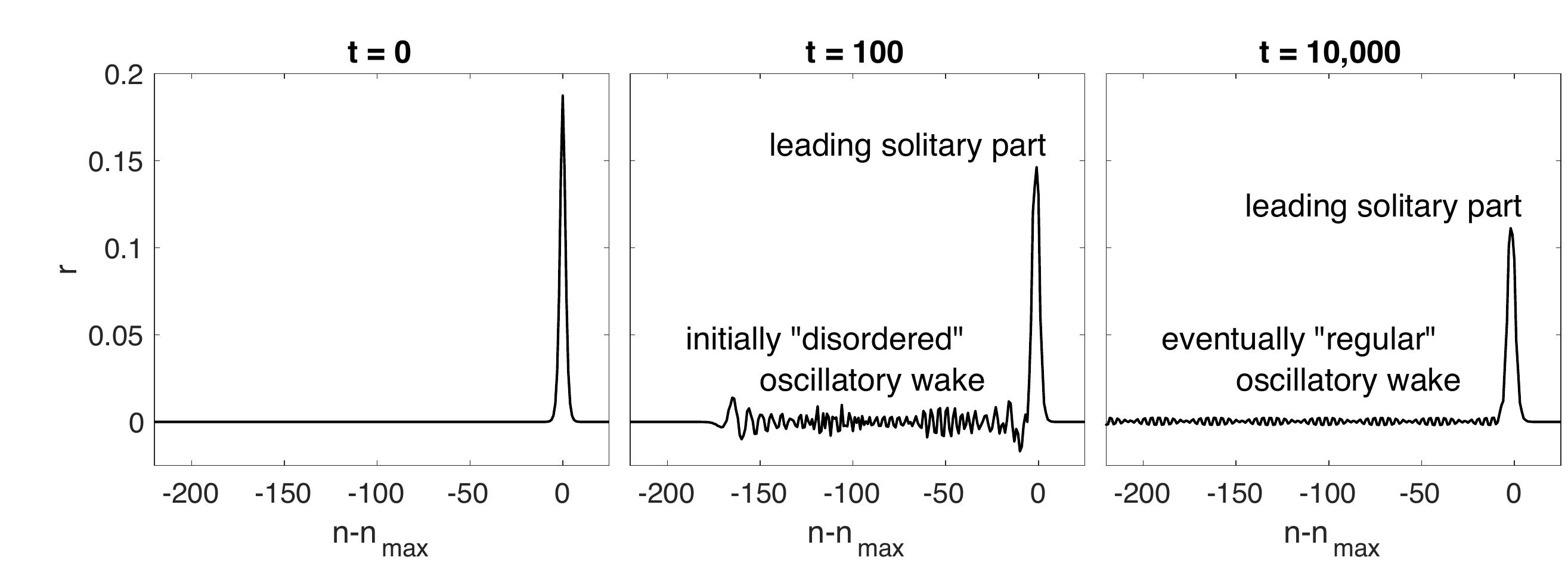}    \caption{An example of a ``radiating solitary wave'' in a diatomic FPUT lattice. The wave is propagating to the right and the horizontal axis is in a moving reference frame.
    Figure taken from \cite{GSWW}.}
    \label{co-ops}
\end{figure}
A similar phenomena occurs in simulations of mass-in-mass lattices \cite{HHMCKKYV, faver-goodman-wright}, as well as in a variety of models for the evolution of capillary-gravity waves \cite{benilov, bona-etal, tan-etal}, though in this latter case the emission runs ahead of the solitary wave. Simulations of water waves over variable bottom topography evoke a similar dynamic \cite{nachbin-etal, nachbin-papanicolaou, liu-yue, nakoulima-etal}. These sorts of waves, often referred to as {\it radiating solitary waves}, are  examples of metastable structures in nonlinear dispersive systems.

The observed 
 attenuation is so slow that one of the major tools for analyzing the dynamics of solitary wave-like solutions, namely approximations of the problem with the Korteweg-de Vries  (KdV) (or similar) equations \cite{GMWZ, schneider-wayne, hadadifard-wright, chirilus-bruckner-etal, pelinovsky-schneider},
is incapable of capturing the phenomena. 
Such approximations are valid over very long but nevertheless finite time intervals; the erosion is so subtle during the period of good approximation that it falls within the natural error bounds. 
Moreover, radiating solitary waves
very often occur in problems where the construction of genuinely localized solitary waves  fails and what is found instead are {\it generalized solitary waves} (also known as a {\it nanopterons}) \cite{beale, akylas-yang, sun, faver-wright, faver-spring, hoffman-wright, faver-mim, lombardi, johnson-wright, joshi-lustri, lustri-porter}.  
These traveling wave solutions are asymptotic at spatial infinity to very small amplitude co-propagating periodic waves and are consequently of infinite energy, further evidence that the dynamics of finite energy solitary wave-like solutions is subtle. 

Unlike their steady counterparts the nanopterons, 
there is at this time no fully rigorous mathematical explanation or description of these radiating solitary waves in any of the many problems in which they arise.
There are some careful non-rigorous treatments (especially \cite{benilov}), but in the main investigations are heuristic or numerical.
The phenomenon is usually attributed to an excitation of some sort of high-frequency oscillation 
by the solitary wave, due, for instance, to band structure considerations \cite{GSWW}, Bragg scattering \cite{liu-yue}, internal resonances \cite{faver-goodman-wright} or non-monotone dispersion relations. 
Notably, the amplitude decays so slowly that in most cases it is very hard to determine the rate with any level of precision and only a few articles  hazard a prediction, not all of which agree. For instance, the formal analysis in \cite{benilov} predicts the decay is exponential with a very small decay rate whereas  \cite{GSWW} (co-authored by the author of this paper) conjectures an algebraic rate of decay based on numerical evidence.

To better understand radiating solitary waves, in this article 
we formulate a rather bare-bones model which captures many essential features at play and which we can rigorously analyze. For our system we are able to exactly pin down the rate at which the solitary wave decays as well as a rather complete description of the radiating tail.
Note that our model is decidedly not meant to quantitatively portray any of the specific systems mentioned but instead provide a partial skeleton for the rigorous analysis of such systems down the line.

Here are the main ingredients of our  model:
\begin{itemize}
\item A spatial variable $x \in \R$ and time variable $t \in \R$.
\item A solitary wave of fixed speed and profile but variable amplitude: $a(t) q(x-t)$.
Here $q(x)$ is the profile and $a(t)\in \R$ is the amplitude. We assume that $q(x)$ is real-valued
and satisfies some decay/regularity conditions we make precise below.
\item  A high-frequency simple harmonic oscillator located at each point $x \in \R$.
We represent each oscillator  by its complex amplitude $\psi(x,t) \in \C$. This field of oscillators is driven by 
the solitary wave in a naive way: 
\be\label{SHO}
 \psi_t = i \om \psi + \om a(t) q(x-t).
\ee
In the above, the oscillators' natural frequency is $\omega \gg 1.$
\item A conserved energy
\be\label{energy}
E : ={1 \over 2} a(t)^2 + {1 \over 2 \om^2} \int_\R \left \vert \psi(x,t) \right \vert^2 dx.
\ee
The first term is proportional to the square of the $L^2$-norm ({\it i.e.}~energy) of the solitary wave and the second is the energy of the oscillator field.
\end{itemize}

Together \eqref{SHO} and \eqref{energy} form a closed system for the variables $a(t)$ and $\psi(x,t)$.
Our main result states that $a(t) \approx a_0 e^{-\theta_\om t}$ where $\theta_\om$ is all but equal to\footnote{The normalization/notation for the Fourier  transform of a function $f: \R \to \C$ we use is: $\ds \Fo[f](k):=\hat{f}(k):={1 \over 2\pi} \int_\R e^{-ikx} f(x) dx$.}
 $2 \pi^2|\hat{q}(\om)|^2$. Weak assumptions on $q(x)$ tell us that $\hat{q}(\om) \to 0$ as $\om \to \infty$ and as such
 the rate of decay is very slow when $\om$ is big.
 In particular, if $q(x)$ analytic then one knows that $|\hat{q}(\om)| \le c_1 e^{-c_2 |\om|}$ for some positive constants $c_1$ and $c_2$.
As such the decay rate of $a(t)$ is incredibly small, beyond all algebraic orders of $\om$. For instance, we find that if $\om = 10$ and $q(x) = \sech(x)$ then the time it will take for the amplitude to deteriorate to $95\%$ of its original amplitude is on the order of $10^{11}$.

We also prove several results about the asymptotics of the oscillator field, the most salient of which is that $e^{\theta_\om t} \psi(x,t)$ converges to a scalar multiple of a specific profile function
$\sigma_\om(x-t)$. This convergence is pointwise in the moving reference frame $x-t$. 
We have an explicit formula for this asymptotic profile which ultimately leads to the description
$$
\sigma_\om(x-t) \approx i q(x-t)/\om + \text{trailing periodic wave of extremely small amplitude}.
$$
That is to say, when viewed in a frame moving along with the solitary wave the oscillator field looks like roughly like a copy of the solitary wave plus a tiny periodic tail, exactly the qualitative behavior seen in the systems we are hoping to understand.

In Section \ref{reformulations} we carry out a number of reformulations of \eqref{SHO} and \eqref{energy} which put the analysis and simulation
 of solutions within grasp. In Section \ref{fate of a} we use the Laplace transform to describe the dynamics of $a(t)$ for large times, the specifics of which are contained in Theorem~\ref{main}.
 In Section \ref{fate of psi} we similarly analyze the oscillatory field $\psi(x,t)$; see Theorem \ref{main2}. Lastly, in Section \ref{sims}, we describe the results of simulations of solutions of our model and compare them with our analytical results. 

{\bf Acknowledgements:} The author is grateful to the NSF who funded this work under grant DMS-2006172.

\section{Reformulations and reductions}\label{reformulations}

\subsection{An ODE on a Banach space}
Differentiation of  \eqref{energy} with respect to time and the use of \eqref{SHO} to eliminate $\psi_t$ get us:
$$
0 = a(t) \dot{a}(t) + {1 \over \om} \int_\R \Re \left[\psi^*(x,t)( i \psi(x,t) + a(t) q(x-t))\right] dx.
$$
The quantity $\psi^* i \psi$ is purely imaginary whereas $a$ and $q$ are purely real 
so we can simplify this to:
\be\label{a eqn}
\dot{a}(t) = - {1 \over \om} \int_\R \Re \left[\psi(x,t)\right] q(x-t) dx.
\ee

Note that now \eqref{SHO} and \eqref{a eqn} are an ODE for the variables $(\psi,a)$ on the Banach space $L^2(\R;\C) \times \R$. Under the rather weak assumption that $q(x) \in L^2(\R;\R)$ it is simple enough to show that the mapping defined by the right hand sides of these equations is smooth from $L^2(\R;\C) \times \R$ into itself and as such Picard's Theorem gives the existence of solutions for short times. The conservation law \eqref{energy} then allows us to extend the solution for all times. In short, we can prove:
\begin{theorem}\label{gwp}
Fix $q(x) \in L^2(\R;\R)$ and $\om > 0$. Then for any $\psi_0(x) \in L^2(\R;\C)$ and $a_0 \in \R$ there exists
unique
$
(\psi,a) \in C^1(\R;L^2(\R;\C) \times \R) 
$
which satisfy \eqref{SHO}, \eqref{a eqn}, $\psi(x,0) = \psi_0(x)$ and $a(0) = a_0$. 
The map carrying $(\psi_0,a_0)$ to $(\psi,a)$ is continuous (for $t$ in compact sets). Additionally for all $t$
$$
{1 \over 2} a(t)^2 + {1 \over 2 \om^2} \int_\R \left \vert \psi(x,t) \right \vert^2 dx =
{1 \over 2} a_0^2 + {1 \over 2 \om^2} \int_\R \left \vert \psi_0(x) \right \vert^2 dx.
$$
\end{theorem}
Before moving on, note that formulating our system as in \eqref{SHO} and \eqref{a eqn} is a convenient starting point for performing simulations, which we do in Section \ref{sims}.


\subsection{A scalar delay differential equation}
We can eliminate the oscillatory field from the governing equations entirely.
We solve \eqref{SHO} by Duhamel's formula:
\be\label{duham1}
\psi(x,t) = e^{ i \om t} \psi_0(x) + \om \int_0^t e^{i \om (t-s)} a(s) q(x-s) ds.
\ee
Putting this into \eqref{a eqn}:
$$
\dot{a}(t) = \bunderbrace{- {1 \over \om}  \int_\R \Re\left[
e^{ i \om t} \psi_0 (x)\right]q(x - t) dx}{j_\om(t)}
-  \int_\R \Re\left[\int_0^t e^{i \om (t-s)} a(s) q(x-s) ds
\right]q(x - t) dx.
$$
Rearranging the order of integration in the second term plus some algebra yields
\be\label{the chosen one}
\dot{a}(t) =j_\om(t)
- \int_0^t \cos( \om (t-s)) q \star q(t-s)a(s)  ds
\ee
where 
$$\ds
q\star q (t) := \int_\R q(x + t) q(x) dx
$$
is the autocorrelation of $q$.
It is worth pointing out here that $q\star q(t)$ is an even function of $t$.

The scalar continuous delay differential equation \eqref{the chosen one} is equivalent to our original system and will be the formulation 
on which we do our analysis. Nevertheless there is one more change we make 
that permits 
a short formal analysis.

\subsection{A scalar renewal equation}
Integrating \eqref{the chosen one} 
 from $0$ to $t$ gives
$$
a(t) = a_0 + \bunderbrace{\int_0^t j_\om(\tau) d\tau}{f_\om(t)}- \int_0^t \int_0^\tau \cos(\om(\tau-s)) q \star q(\tau-s) a(s) ds d\tau.
$$
Exchanging the order of integration in the second term leads to
%
%
\be\label{renewal}
a(t) = a_0 + f_\om(t) + \int_0^t \phi_\om(t-s) a(s) ds
\ee
where
$$
\phi_\om(t) :=-\int_0^t \cos(\om s) q \star q(s)ds.
$$
Equation \eqref{renewal} is a renewal equation of the sort studied in \cite{bellman-cooke}, though in most applications of renewal equations the function $\phi_\om(t)$ is positive with finite first moment (neither of which is the case for us).

We do a quick non-rigorous analysis of \eqref{renewal} now, in the case where $f_\om(t) =0$.
Since $q \star q(t)$ is even, the Fourier convolution theorem implies $\phi_\om(t) \to -2 \pi^2|\hat{q}(\om)|^2$
as $t \to \infty$. If we simply replace $\phi_\om$ with this asymptotic state in \eqref{renewal} we have
$
a(t) = \ds a_0  -2 \pi^2|\hat{q}(\om)|^2\int_0^t a(s) ds.
$
Differentiation of this gives $\dot{a} = -2 \pi^2 \pi^2|\hat{q}(\om)|^2 a$ and so 
the solution of this simplified equation is $a_0 e^{-2 \pi^2|\hat{q}(\om)|^2t}$. This gives us our  first glimpse as to what the rate of decay for $a(t)$ is going to be. 
Of course this replacement of $\phi_\om$ is completely unjustified which is why we now move on to rigorous analysis.

\section{The fate of $a(t)$} \label{fate of a}
In this section we determine the long time asymptotics of $a(t)$ using \eqref{the chosen one} as the starting point.
The integral on the right hand side of that equation is a time-domain convolution of $a(t)$ and 
$$
k_\om(t):=\cos(\om t) q \star q(t).
$$
So if we apply the Laplace transform\footnote{Here, we use the notation that the Laplace
transform of a function $f: \R^+ \to \C$ is $\L[f](z) := F(z) := \int_0^\infty e^{-zt} f(t) dt.$} to \eqref{the chosen one} we get
\be\label{THE CHOSEN ONE}
z A(z) - a_0 = J_\om(z) -  K_\om(z) A(z).
\ee
We have made use of a number of well-known properties of the Laplace transform here. We adhere
to the usual convention that 
the Laplace transform of a variable whose name is in lower case is represented by the corresponding capital.


 We isolate $A(z)$ in \eqref{THE CHOSEN ONE}:
\be\label{here is A}
A(z) = {a_0 \over z+ K_\om(z)}  + {J_\om(z) \over z+K_\om(z)}.
\ee
From this we ascertain that the 
 key to understanding the evolution of $a(t)$ is the function $v_\om(t)$ whose Laplace transform 
is 
\be\label{this is V}
V_\om(z):={1 \over z+K_\om(z)}.
\ee
The following result tells us everything we need to know about $V_\om(z)$:
\begin{lemma} \label{key lemma} Suppose that, for some $\rho >0$, $e^{\rho t} q \star q (t)\in W^{1,\infty}(\R^+;\R)$.
Then, for all $\omega$ sufficiently large, $V_\om(z)$ has exactly one 
pole in the set $\Re(z) \ge - \rho/2$. This pole is simple and located at $z = -\theta_\om \in \R$ where
 \be\label{theta}
 \theta_\om =  2 \pi^2 |\hat{q}(\om)|^2(1+\O(1/\om)) .
 \ee
\end{lemma}

\begin{proof}
{\bf Preliminary estimates:}
For functions $f: \R^+ \to \C$ we put:
$$
\| f \|_\rho:= \sup_{t \ge 0} e^{\rho t} |f(t)|.
$$
For functions with $\| f \|_\rho < \infty$ we have the following elementary estimate, which holds for $\Re(z) > -\rho$:
\be\label{elementary}
|\L[f](z)|=\left\vert \int_0^\infty e^{-zt} f(t) dt \right\vert \le \int_0^\infty e^{-\Re(z) t} \|f\|_{\rho} e^{-\rho t} dt
\le {1 \over \Re(z) +\rho}\|f\|_{\rho}.
\ee
If we have $\| \dot{f}(t) \|_\rho < \infty$ then an integration by parts in the above gets:
\be\label{not so elementary}\begin{split}
|\L[f](z)|  
= &\left \vert 
{1 \over z} f(0) + {1 \over z} \int_0^\infty e^{-zt} \dot{f}(t) dt
\right \vert 
\le {1 \over |z|} \left( |f(0)|+ {1 \over \Re(z) +\rho}\|\dot{f}\|_{\rho}\right).
\end{split}\ee
This again holds for $\Re(z) > - \rho$. Lastly, the famous identity $\ds \L[t f](z) = - {d \over dz} \L[f] (z),$ combined with \eqref{not so elementary} implies
\be\label{nse deriv}
\left \vert {d \over dz} \L[f](z)\right \vert \le {1 \over |z|}  {1 \over \Re(z) +\rho}\left\|{d \over dt} \left(t f\right)\right\|_{\rho}.
\ee

\noindent
{\bf There are no poles far away from the origin:}
Note that the poles of $V_\om(z)$ are precisely the zeros of $z + K_\om(z)$ and consequently our analysis
will focus on $K_\om(z)$. 
The assumptions on $q \star q$ and the definition of $k_\om$ imply that $K_\om(z)$ is analytic in the set $\Re(z) > -\rho$.
Moreover we have 
$\| k_\om \|_\rho  \le \| q \star q\|_\rho<\infty$. Importantly, $\|k_\om\|_\rho$ can be controlled independently of $\om$.
Using \eqref{elementary} we have
$
|K_\om(z)| \le \| q \star q\|_\rho/(\Re(z) + \rho).
$
The restriction that $\Re(z) \ge -\rho/2$ tells us that
$|K_\om(z)| \le 2 \| q \star q\|_\rho /\rho$. Therefore
$
z+K(z) \ne 0$ when  $|z| > 2 \| q \star q\|_\rho/\rho$ and the only place to look for the zeros is in inside the
set $$
U_\rho:=\left \{ |z| \le 2 \| q \star q\|_\rho /\rho\right\} \cap \left\{\Re(z) \ge -\rho/2\right\}.$$

\noindent
{\bf There is just one pole near the origin:}
We use the fact that $\cos(\om t)$ is of high-frequency; the frequency shifting formula for the Laplace transform tells us
\be\label{K thing}
K_\om(z) = {1 \over 2} \L[q \star q](z+i \om) +  {1 \over 2} \L[q \star q](z-i \om).
\ee
Using \eqref{not so elementary} on the right hand side we have:
\be\label{K is small}
|K_\om(z)| \le \left({1 \over 2 |z+i\om| } +{1 \over 2 |z+i\om| }\right)\left( q \star q (0) + {1 \over \Re(z) + \rho} \left\| {d \over dt} q \star q\right\|_{\rho} \right).
\ee

We have assumed that $\left\| ({d/dt}) q \star q\right\|_{\rho} < \infty$ and this quantity is clearly independent of $\om$, as is $q\star q(0)$. In $U_\rho$ we have $\Re(z) + \rho \ge \rho/2$ and $|z \pm i \om| \ge \om - |z| \ge \om - 2 \| q \star q\|_\rho /\rho$. All these together tell us that there are constants $\om_0>0$ and $C_0>0$ for which 
\be\label{K est}
\om \ge \om_0 {\text{ and }} z \in U_\rho \implies |K_\om(z)| \le {C_0/\om}.
\ee
Now let $r:=\min\left\{ |z| : z \in \delta U_\rho\right\}>0$, which is independent of $\om$. Thus the preceding estimate allows us to find $\om_1>0$ so that 
$$
\om > \om_1 {\text{ and }} z \in U_\rho \implies |K_\om(z)| \le r/2.
$$

Thus we have $|K_\om(z)| < |z|$ on the boundary of $U_\rho$. Since $z$ and $K_\om(z)$ are analytic in $U_\rho$ we conclude, by Rouche's Theorem, that $z$ and $z+K_\om(z)$ have the same number of zeros (counted with multiplicity) inside $U_\rho$. Which is to say $z + K_\om(z)$ has one simple zero in that set.

And so, as foretold in Lemma \ref{key lemma}, we have shown that $V_\om(z) = 1/(z+K_\om(z))$ has exactly one simple pole (whose
location is denoted by $-\theta_\mu$) in the set $\Re(z) \ge -\rho/2$. What remains is to further pin down this pole as described in \eqref{theta}. 

\noindent{\bf Characterizing the pole:}
Since $q \star q(t)$ is even in $t$ we have
\be\label{K0}
K_\om(0) = \int_0^\infty \cos(\om t) q \star q (t) dt =
\pi \Fo[q \star q](\om)=2\pi^2 |\hat{q}(\om)|^2.
\ee
The final equality is due to the convolution theorem. 
The regularity/decay condition on $q\star q(t)$ implies 
$\Fo[q \star q](\om) \to 0$ as $\om \to \infty$, and thus so does $K_\om(0)$.

Differentiating \eqref{K thing} we have:
\bes
K'_\om(z) = {1 \over 2} {d \over dz} \L[q \star q](z+i \om) +  {1 \over 2} {d \over dz} \L[q \star q](z-i \om).
\ees
Then we use \eqref{nse deriv} to get
$$
|K'_\om(z)| \le {1 \over 2}  \left( {1 \over |z+i\om|} +{1 \over |z-i\om|} \right) {1 \over \Re(z) +.9 \rho}\left\|{d \over dt} \left(t q \star q\right)\right\|_{.9\rho}.
$$
Note that we have used the fact that $\|  f\|_\rho < \infty$ implies $\| tf \|_{\rho'} < \infty$ for $\rho'<\rho$.
Much as in the run up to \eqref{K est}, the above estimate implies the existence of $\om_2>0$ and $C_2>0$ such that
\be\label{K' est}
\om \ge \om_2 {\text{ and }} z \in U_\rho \implies |K'_\om(z)| \le {C_2/\om}.
\ee

The end approaches. Note that $K(z)$ is real-valued if $z=s \in \R$. The fundamental theorem of calculus implies that 
$
|K_\om(s) -K_\om(0)| = \left \vert \int_0^s K'(s')ds' \right \vert
$
and so \eqref{K' est} leads to:
$$
K_\om(0) + (1-C_2/\om)s \le s + K_\om(s) \le  K_\om(0) + (1+C_2/\om)s.
$$
The function on the right is zero at $s=-K_\om(0)/(1+C_2/\om)$ and the function on the left at $s=-K_\om(0)/(1-C_2/\om)$.
The intermediate value theorem then tells us that $s+K_\om(s)$ has a zero somewhere between these two points, which, because $K_\om(0) \to 0$ as $\om \to \infty$, are both in $U_\rho$.
Consequently that zero 
is $-\theta_\om$, the unique zero of $z+K_\om(z)$ in $U_\rho$ we found earlier.
All this, and \eqref{K0}, give:
$$
{2\pi^2 |\hat{q}(\om)|^2 \over 1 + C_2/\om} \le \theta_\om \le {2\pi^2 |\hat{q}(\om)|^2 \over 1 - C_2/\om}. 
$$
This is equivalent to \eqref{theta} and the proof is complete.

\end{proof}

Lemma \ref{key lemma} tell us that the rightmost pole of $V_\om(z)$ is located at $-\theta_\om$. The common wisdom that
pole placement determines decay rate  tells us to expect $v_\om(t)=\L^{-1}[V_\om](t)$ will behave like $e^{-\theta_\om t}$ as $t \to \infty$. Nevertheless inverting the Laplace transform is a sometimes subtle business and so we make a precise
statement and proof.
\begin{theorem}\label{main}
Suppose that, for some $\rho >0$, $e^{\rho |x|} q(x)\in W^{1,\infty}(\R;\R)$. 
Then the the solution of \eqref{SHO} and \eqref{a eqn}
with initial data $\psi(x,0) = 0$ and $a(0) = a_0 $ satisfies (for $\om$ sufficiently large)
\be\label{first}
a(t) = a_0 r_\om e^{-\theta_\om t} + b_\om(t)
\ee
where $\theta_\om = 2\pi^2 |\hat{q}(\om)|^2 (1+ \O(1/\om))$, $r_\om = 1 + \O(1/\om)$ and $\| b_\om \|_{\rho'} < \infty$
for any $\rho' < \rho/2$.

Additionally if $e^{\rho|x|} \psi_0(x) \in L^\infty(\R;\C)$ then the solution
of \eqref{SHO} and \eqref{a eqn} with initial data $\psi(x,0) = \psi_0(x)$ and $a(0) = a_0 $ satisfies (for $\om$ sufficiently large)
\be\label{second}
a(t) = \alpha e^{-\theta_\om t} + {b}_\om(t)
\ee
for some constant $\alpha = \alpha(\psi_0,a_0) \in \R$; $b_\om(t)$ satisfies the same estimate as above.

\end{theorem}

\begin{proof}
From Lemma \ref{key lemma} we know that the pole of $V_\om(z)$ at $z = -\theta_\mu$ is simple. We now compute
the residue at the pole in the usual way:
$$
r_\om:=\textrm{Res}(V_\om,-\theta_\om) = \lim_{z \to -\theta_\om} (z+\theta_\om) V_\om = \lim_{z \to -\theta_\om} {z + \theta_\om \over z+ K_\om(z)} = {1 \over 1 +K'(-\theta_\om)}.
$$
Estimate \eqref{K' est} tells us that $r_\om = 1 + \O(1/\om)$.

With this we now know that
$$
W_\om(z):=V_\om(z) - {r_\om \over z+ \theta_\om}
$$
is analytic for $\Re(z) > -\rho/2$. 
Because of \eqref{K is small} we have
$$
\lim_{z \to \infty} z W_\om(z) = {z \over z + K_\om(z)} - {r _\om z \over z + \theta_\om} = 1 - r_\om
$$
provided $\Re(z) > -\rho/2$. Thus we have $|W_\om(z)| \le C/|z|$ for $|z|$ big enough. This is not a rapid of enough decay to for us
to deploy standard inversion results from (for instance) \cite{korner}.

To get around this we let
$$
B_\om(z) := V_\om(z) - {r_\om \over z+ \theta_\om}- {1 - r_\om \over z + 2\rho}.
$$
The final term there is a sort of ``fudge factor.''  Note that $B_\om(z)$ is analytic for $\Re(z)> -\rho/2$. It is easy to see that $\lim_{z \to \infty} z B_\om(z) =0$ provided $\Re(z) > -\rho/2$. 
Moreover a routine computation gives:
\be\begin{split}
&\lim_{z \to \infty} z^2 B_\om(z)
 = (-2r_\om + 2)\rho + r_\om \theta_\om .
 \end{split}\ee
 Thus  $|B_\om(z)| < C/|z|^2$ for $|z|$ big enough. This rate is fast enough to use Lemma 76.4 in \cite{korner} and conclude that
 $e^{\rho'  t}\L^{-1}[B_\om](t) \to 0$ as $t \to \infty$ for any $\rho' < \rho/2$.
 
 And so all together we find that
 \bes\begin{split}
 v_\om(t)
 = \L^{-1}\left[ {r_\om \over z+ \theta_\om}+ {1 - r_\om \over z + 2\rho}+B_\om\right](t)=
 r_\om e^{-\theta_\om t} + (1-r_\om) e^{- 2 \rho t} + \L^{-1}[B_\om](t).
 \end{split}
 \ees
 If $\psi_0(x,0)=0$ then from \eqref{here is A} we have $a(t) = a_0 v_\om(t)$ and the above leads directly to \eqref{first}.

 Now for \eqref{second}. If $\psi_0(x) \ne 0$ 
 the decay conditions placed upon it imply that $\| j_\om\|_{\rho} \le C/\om$; here is the calculation:
 $$
 |e^{\rho t} j_\om(t)| \le {1 \over \om}  e^{\rho t}\int_\R |\psi_0(x)| |q(x-t)| dx
 \le{C \over \om}  e^{\rho t}\int_\R e^{-\rho|x|} e^{-\rho|x-t|} dx \le {C \over \om}.
 $$
This tells us that $J_\om(z)$ is analytic for $\Re(z) > -\rho$. Thus the term $J_\om(z) V_\om(z)$ in \eqref{here is A}
is analytic in $\Re(z) > -\rho/2$ except for the pole of $V_\om(z)$ at $-\theta_\om$. The same sorts of steps as above can be repeated to show that 
 $\ds
 \lim_{t \to \infty} e^{\theta_\om t} \L^{-1}[J_\om  V_\om] (t)
 $
 exists, from which \eqref{second} follows. 
 \end{proof}
 
 \section{The fate of $\psi(x,t)$}\label{fate of psi}
 
Now that we have determined the dynamics of $a(t)$ for large values of $t$, we do the same for $\psi(x,t)$. We begin by
observing
that because $a(t) \to 0$ and we have the conservation of the energy \eqref{energy} we know that $\psi (x,t)$ does not converge to zero
in the $L^2(\R;\C)$ norm. This is not much of a statement, but it does indicate that the ultimate behavior of $\psi(x,t)$ is 
not disintegration.

For a more refined analysis, our starting point is
\eqref{duham1} which expresses $\psi(x,t)$ explicitly in terms of $a(t)$. 
Here is our first result:
\begin{cor}
 If $e^{\rho |x|} q(x)\in W^{1,\infty}(\R;\R)$ and $e^{\rho|x|} \psi_0(x) \in L^\infty(\R;\C)$ then the solution
of \eqref{SHO} and \eqref{a eqn} with initial data $\psi(x,0) = \psi_0(x)$ and $a(0) = a_0 $ satisfies (for $\om$ sufficiently large)
$$
\lim_{t \to \infty} e^{-i \om t} \psi(x,t) = \psi_0(x) + \om \int_0^\infty e^{-i \om s} a(s) q(x-s) ds
$$
pointwise in $x$.
\end{cor}
\begin{proof}
Divide both sides of \eqref{duham1} by $e^{i \om t}$ and take the limit. The integral converges due to the restrictions placed upon $\psi_0$ and $q$.
\end{proof}
The most important takeaway from this result is that it indicates that $\psi(x,t)$ does not decay to zero (in the supremum norm)
in the large time limit. 
This can be make rigorous by drilling 
down into the integral term above  to get a more refined picture, but it turns out it is more interesting to view $\psi(x,t)$ in a frame that moves along with the solitary wave $q(x-t)$.  And so we put $x-t =l$ and 
$
\varphi(l,t) = \psi(x,t).
$
This converts  \eqref{duham1} to
\be\label{duhammove}
\varphi(l,t) = e^{i \om t} \psi_0(t+l) + \om \int_0^t e^{i (t-s)} a(s) q(t+l-s) ds.
\ee

We begin by considering the situation where $\psi_0(x) = 0$ and $a_0 = 1$, in which case $\varphi(l,t) = \gamma_\om(l,t)$ with
$$
\gamma_\om(l,t) :=  \om \int_0^t e^{i \om (t-s)} v_\om(s) q(t+l-s) ds.
$$
If we fix $l$ and take the Laplace transform of the above with respect to $t$, the convolution and frequency shifting identities get 
us
$$
\Gamma_\om(l,z) = \om Q_l(z-i \om) V_\om(z)
$$
where $$
Q_l(z):=\L[q(\cdot+l)](z) .
$$ With our standard assumption that  $e^{\rho|x|} q(x) \in W^{1,\infty}(\R;\R)$ we have
 $\|q(\cdot+l)\|_\rho < \infty$ for any $l \in \R$. Thus $Q_l(z)$ will be analytic when $\Re(z)>-\rho$. In turn  this implies that $\Gamma_\om(l,z)$ will inherit the simple pole at $z = -\theta_\om$ from $V_\om(z)$ and that this is the only singularity when $\Re(z) \ge -\rho/2$.
 It is easy to compute that
\be\label{sigma}
\textrm{Res}(\Gamma_\om(l,z),-\theta_\om) = \om r_\om\bunderbrace{ Q_l(-\theta_\om - i\om)}{\sigma_\om(l)}.
\ee
Thus the common wisdom implies that $\gamma_\om(l,t)$ looks like $\sigma_\om(l) e^{-\theta_\om t}$ as $t \to \infty$. We codify this in the following result, whose proof is so close to that of Theorem \ref{main} we omit it:
\begin{theorem}\label{main2}
Suppose that, for some $\rho >0$, $e^{\rho |x|} q(x)\in W^{1,\infty}(\R;\R)$. 
Then the the solution of \eqref{SHO} and \eqref{a eqn}
with initial data $\psi(x,0) = 0$ and $a(0) = a_0 $ satisfies (for $\om$ sufficiently large)
\be\label{third}
 \psi(t+l,t) = a_0 \om r_\om \sigma_\om(l)e^{-\theta_\om t} + \eta(l,t).
\ee
Here $\sigma_\om(l)$ is given in \eqref{sigma}, $r_\om = 1+\O(1/\om)$ and $\|\eta(l,\cdot)\|_{\rho'} < \infty$ for all $\rho' < \rho/2$ and every $l \in\R$.

Additionally if $e^{\rho|x|} \psi_0(x) \in L^\infty(\R;\C)$ then the solution
of \eqref{SHO} and \eqref{a eqn} with initial data $\psi(x,0) = \psi_0(x)$ and $a(0) = a_0 $ satisfies (for $\om$ sufficiently large)
\be\label{fourth}
 \psi(t+l,t) = \beta \om r_\om \sigma_\om(l) e^{-\theta_\om t}+ \eta(l,t)
\ee
for some finite constant $\beta = \beta(\psi_0,a_0)\in \R$; $r_\om$ and $\eta(l,t)$ satisfy the same estimates as above.
\end{theorem}

\begin{remark} The result above is pointwise in $l$.
It is quite possible that a stronger mode of convergence holds here, though the technical difficulty
in establishing this is motivation enough to leave that for another article.
\end{remark}

To close out this section, we now describe  the ``asymptotic profile'' $\sigma_\om(l)$ in greater detail.
The analysis here is formal though it could be made rigorous if we stack enough hypotheses on $q(x)$.
Recall that $\sigma_\om(l)=
Q_l(-\theta_\om - i \om).  
$
Since $\theta_\om \to 0$ as $\om \to \infty$ and $Q_l(z)$ is analytic in $z$ we have
$
Q_l(\theta_\om - i \om) = Q_l(- i \om) + \O(\theta_\om).
$
Then by definition we have
$$
Q_l(-i \om) = \int_0^\infty e^{i \om t} q(t + l) dt = e^{- i \om l} \int_{l}^\infty e^{i \om x} q(x) dx.
$$
And so we see that 
$
e^{i \om l} Q_l(-i \om) \to 2\pi \hat{q}(-\om)
$
as $l\to-\infty$ and $e^{i \om l} Q_l(-i \om) \to 0$ as $l \to \infty$. This latter convergence will be exponentially fast. 
Lastly the same sort of calculation that led to \eqref{not so elementary} gives
$
Q_l(-i \om) = i q(l)/\om +\O(1/\om^2).
$

Putting everything together we have
$$
\sigma_\om(l) \sim
\begin{cases}
 2\pi  \hat{q}(-\om) e^{- i \om l} & l \ll 0 \\
 {i }  q(l)/\om & l \sim 0\\
 \text{exponential decay} & l \gg 0.
\end{cases} 
$$
In short, $\sigma_\om(l)$ looks like a scalar multiple of  the solitary wave profile plus a trailing periodic wave of small amplitude and frequency $\om$.

\section{Simulations}\label{sims}

We have simulated solutions of our model with a variety choices of the profile $q(x)$ and frequency $\om$.
We always take $\psi_0(x) = 0$ since our analytic results indicate the effects of this part of the initial data are transient and do not alter the long time behavior (or at least the rate of decay of $a(t)$). Likewise $a_0 =1$ in all cases. 

Our method is straightforward:
we treat the system as an ODE for $(\psi,a)$ as in
 \eqref{SHO} and \eqref{a eqn} and simulate using an RK4 algorithm. We implement the integrals in \eqref{a eqn} via Simpson's rule.
We also compute $\sigma_\om(l)$ numerically, as for most functions closed expressions are hard to obtain. We again use Simpson's rule for the computation. All simulations were done in MATLAB.

Before we get into the results we note that with our method high accuracy/long time simulations of the problem are challenging to obtain even for modestly large values of $\om$. Roughly speaking, to accurately resolve a decay rate like $e^{- \theta_\om t}$ we would need to simulate out to times of $\O(1/\theta_\om)$ and (for the RK4 method we use) have a temporal step size which is $\O(\theta_\om)$. The long time of integration in turn 
 implies a large spatial domain (also $\O(1/\theta_\om)$) is needed, since $q(x-t)$ propagates in space. And we need to resolve that spatial domain at the same step size as the temporal one. All these considerations tell us that we need $\O(1/\theta_\om^4)$ operations, at a bare minimum, to experimentally determine $\theta_\om$ in a quantitatively reliable fashion. And in the most interesting cases, $\theta_\om$ is exponentially small in $\om$, meaning that we quickly reach a computational bottleneck.

In light of these inherent difficulties we simply integrate out to $t = 1000$ and choose our step size to be $\pi(10 \om)^{-1}$.
For the largest values of $\om$, our simulations have been pushed past the point at which we can be convinced of their quantitative reliability and instead we view them as a qualitative illustration of our results (which are, after all, fully justified) and the phenomena they describe.

We show our results in the figures which follow. Each figure contains the same four sorts of graphs.
\begin{itemize}
\item {\it Upper left panel:}
A semilog plot of the numerical solution $a(t)$ vs $t$ for the entirety of the run. In the same panel is shown the graph of $\alpha e^{-2\pi |\hat{q}(\om)|^2 t}$ where $\alpha$ is just a scaling factor used to make the graph readable. Theorem \ref{main} tells us that these should be nearly parallel when $\om$ is large and in fact we do see this. 
\item  {\it Upper right panel:} A plot of $a(t)$ vs $t$ during the very beginning of the simulation. In each case we see that the $a(t)$ oscillates a few times and quickly ``settles down'' into the slow decay. The time it takes to settle down does not seem to depend on $\om$, as predicted by Theorem \ref{main}. 
\item  {\it Lower left panel:} The graphs of the real and imaginary parts of $\psi$ vs $x$, at $t = 1000$, the end of the run. The figure is zoomed in on the leading edge of the solution, located near $x = 1000$.  
\item {\it Lower right panel:} The numerically computed asymptotic profile function $\sigma_\om(l)$ vs $l$. The horizontal scale is arranged to match that of the previous panel. Theorem \ref{main2} tells us that $\psi(x,t)$ should look much like (a scalar multiple of) $\sigma_\om(l)$ as $t \to \infty$ and indeed we see exactly this.
\end{itemize}

\subsection{$q(x) = \sech(x)$} We lead off with this choice for $q(x)$ because of the ubiquity of hyperbolic secant in profiles for solitary waves.
One has $\hat{q}(\om) = \sech(\pi \om/2)/2$ and consequently we have
$$
\theta_\om = {\pi^2 \over 2} \sech^2\left( {\pi \om \over 2}\right) (1 + \O(1/\om)).
$$
We take $\om = 2,\ 4$ and $8$. Results are shown in Figures \ref{sechfig2}-\ref{sechfig8}.  
When $\om = 8$ both the decay of $a(t)$ and the trailing oscillatory tails of $\psi(x,1000)$ and $\sigma_\om(l)$ have vanished to the naked eye.

\subsection{$q(x) = e^{-x^2}$}
We selected this because the ultra-rapid decay of its Fourier transform renders the decay rate incredibly small
even at modest values of $\om$. We have
$$
\theta_\om = {\pi \over 2} e^{-\om^2/2}(1 + \O(1/\om)).
$$
We show results for $\om = 4$ and $\om = 6$, in Figures \ref{gaussianfig4} and \ref{gaussianfig6}. Note how there is essentially no decay even at $\om = 6$. Likewise at this value the trailing oscillatory waves in $\psi(x,1000)$ and $\sigma_\om(l)$ are invisible.

\subsection{$q(x) = e^{-|x|}$.}
This ``peakon'' profile 
 is not analytic and consequently the decay rate merely goes to zero algebraically fast and as such the decay is more obvious at larger values of $\om$. To wit
 $$
 \theta_\om = {2 \over (1+\om^2)^2} (1+ \O(1/\om)).
 $$
We show results for $\om = 4,\ 8$ and $16$ in Figures \ref{peakonfig4}-\ref{peakonfig16}.
Both the decay and the tails are visible at $\om = 16$.

\subsection{$q(x) = (1-|x|)_+$.} The subscript ``$+$'' mean to take the positive part, which is to say that $q(x)$ is the ``tent'' map.
This is also non-analytic and its Fourier transform is sometimes zero (unlike the others) and as such if we fine tune $\om$ we can get solutions which do not decay at all. Specifically we have
 $$
 \theta_\om = {1 \over 2} \textrm{sinc}(\om/2)^4
  $$
  so if we take $\om$ to be an even multiple of $\pi$ we should see no decay/no tails. We show results here (Figures \ref{tentfig2pi}-\ref{tentfig5pi}) for 
  $\om = 2 \pi,\ 3 \pi,\ 4\pi$ and $5\pi$ and we see exactly this behavior.

\section{Conclusions, remarks and future directions}

A key takeaway of this article is that in this simple model of radiating solitary waves the rate of attenuation in the amplitude
is very slow, but nevertheless exponential. This is in line with the results of \cite{benilov}  and not 
the algebraic decay rate the author predicted in \cite{GSWW}. The extremely slow rate of decay predicted by the main results provides further evidence that radiating solitary waves present challenging complications both on the numerical and analytic sides. Indeed, even in this simple model really capturing the rate of decay numerically for large values of $\om$ would require a much more sophisticated approach than is used here.

One can reasonably ask, however, whether or not our model here will truly be reflective of the systems which possess radiating solitary waves. After all, the derivation of  \eqref{SHO}-\eqref{energy}  is {\it ad hoc} and (as is always the case in such models) important features/considerations have been omitted. Indeed, in our model only the amplitude of the solitary wave is variable and it has fixed speed. In nearly all systems with solitary waves, the speed, amplitude {\it and wavelength} of the solitary wave are linked. This is particularly important in the proofs of stability of solitary waves in KdV \cite{pego-weinstein} and FPUT \cite{friesecke-pego} and in the derivation of effective equations for solitary waves in potentials \cite{holmer}. By restricting  to a fixed speed and width, our model is ultimately linear in its unknowns and this in turn makes the analysis by Laplace transform possible. Preliminary attempts at incorporating variable speed/wavelength result in nonlinear systems and, consequently, there are substantial technical challenges that make the application of the methods used here non-obvious. One avenue is to adapt the refined asymptotics for stationary problems in \cite{akylas-yang} to this time-dependent setting. Success in that venture will lead to the next, and most exciting challenge:
 connecting such results rigorously to a full system with radiating solitary waves. Work is underway.

  \begin{figure}
\centering
    \includegraphics[width=.75\textwidth]{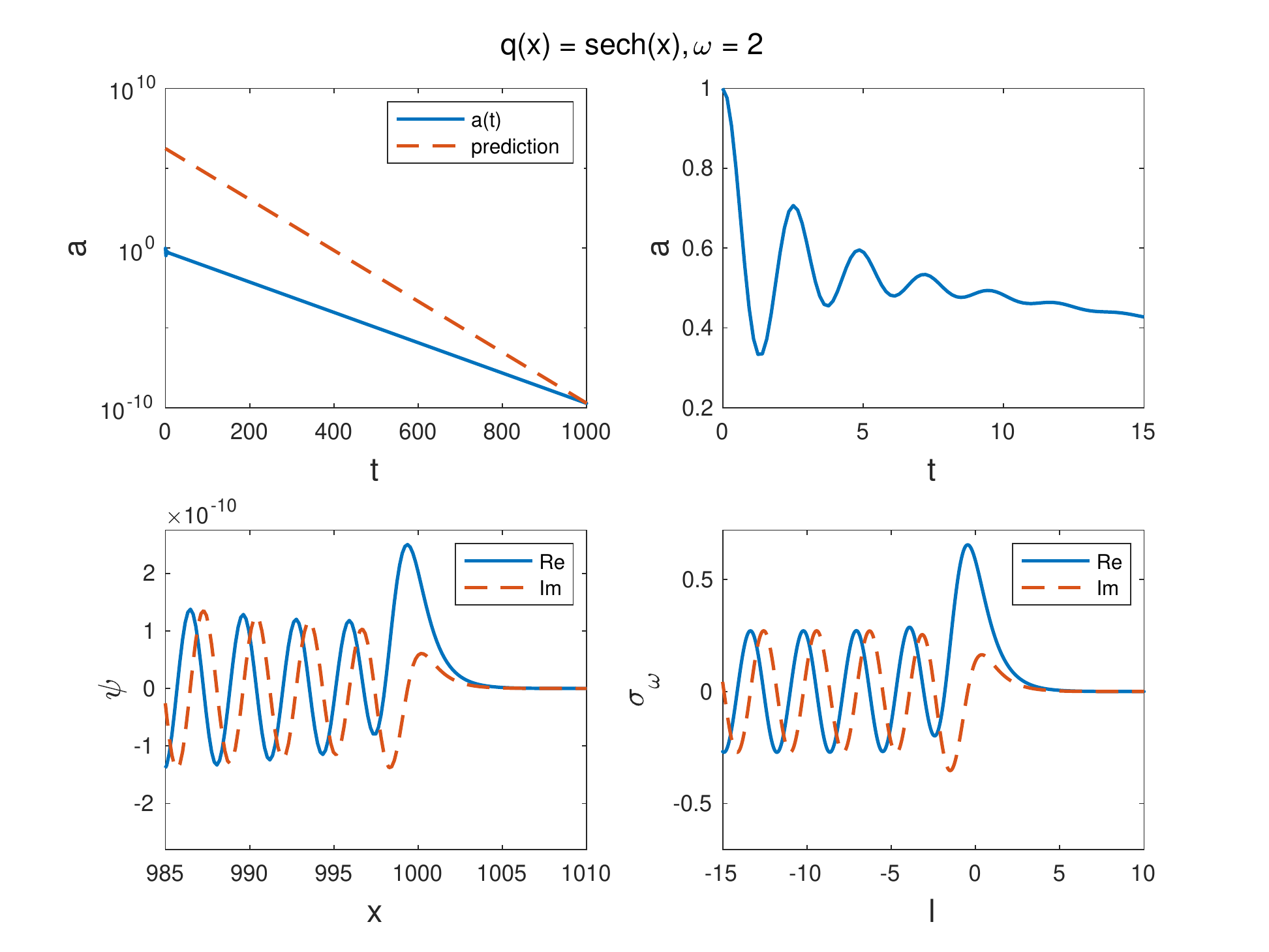}
    \caption{}
    \label{sechfig2}
\end{figure}
\begin{figure}
\centering
    \includegraphics[width=.75\textwidth]{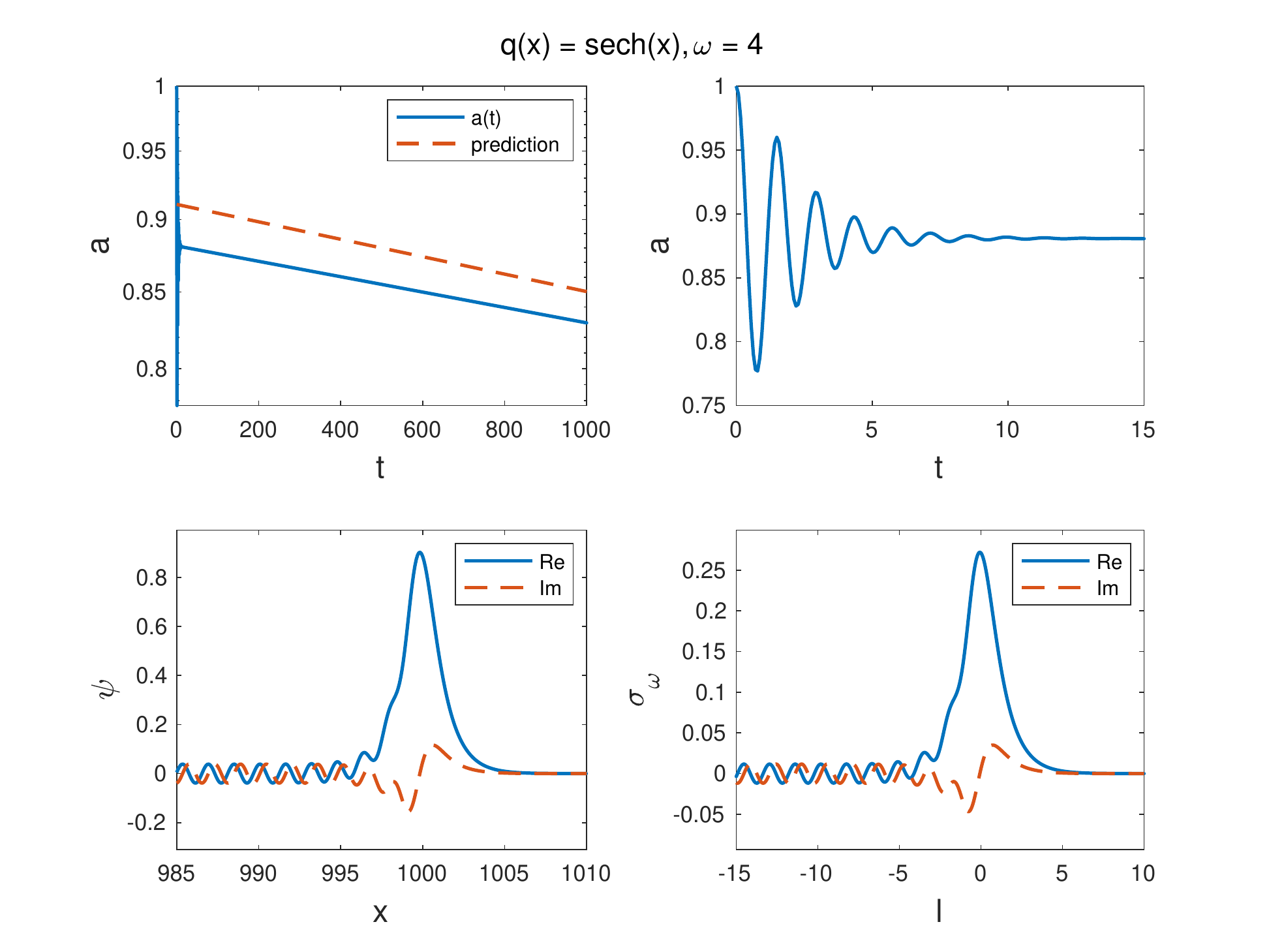}    \caption{}
      \label{sechfig4}
\end{figure}
\begin{figure}
\centering
    \includegraphics[width=.75\textwidth]{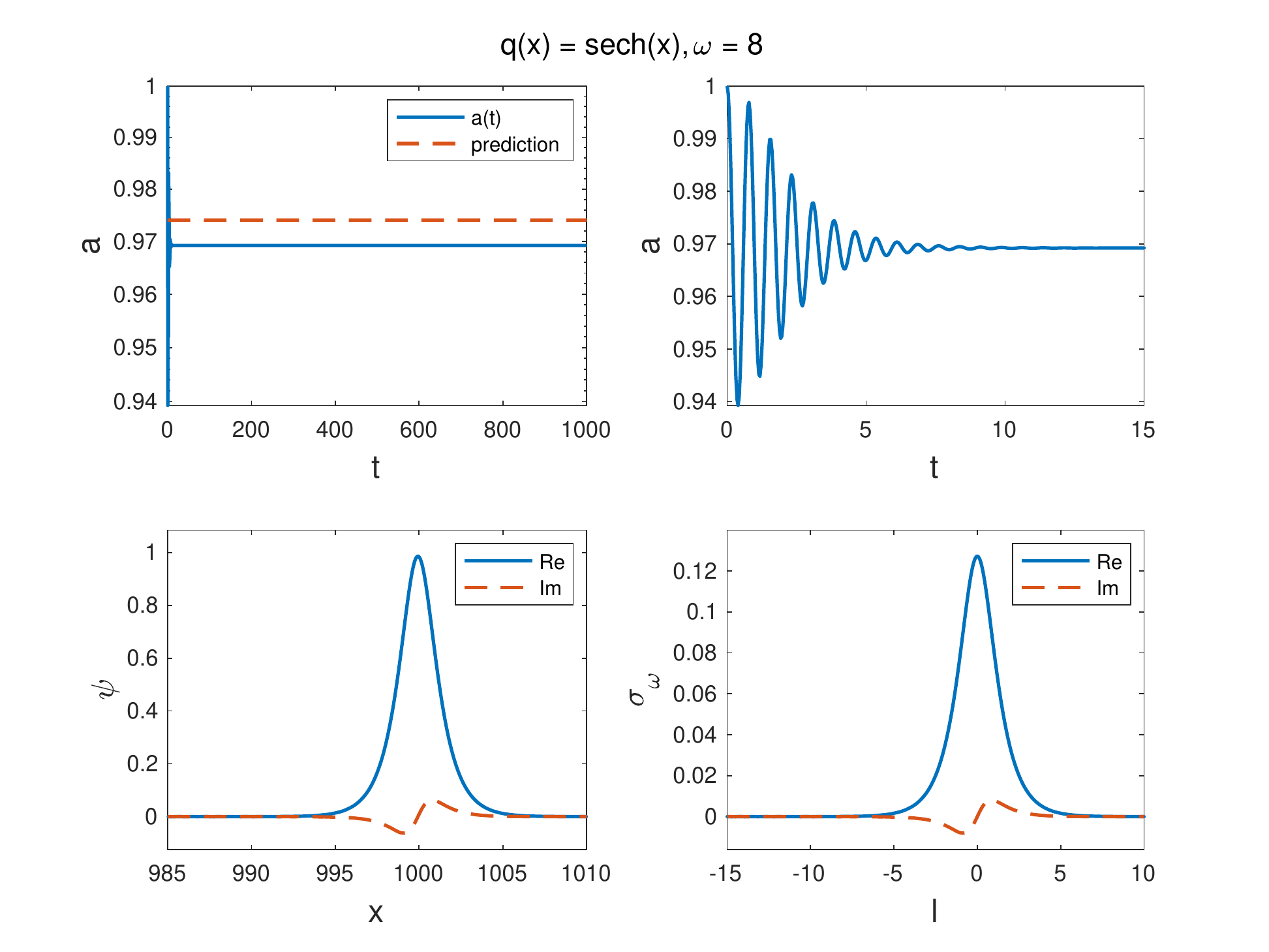}    \caption{}
    \label{sechfig8}
\end{figure}

  \begin{figure}
\centering
    \includegraphics[width=.75\textwidth]{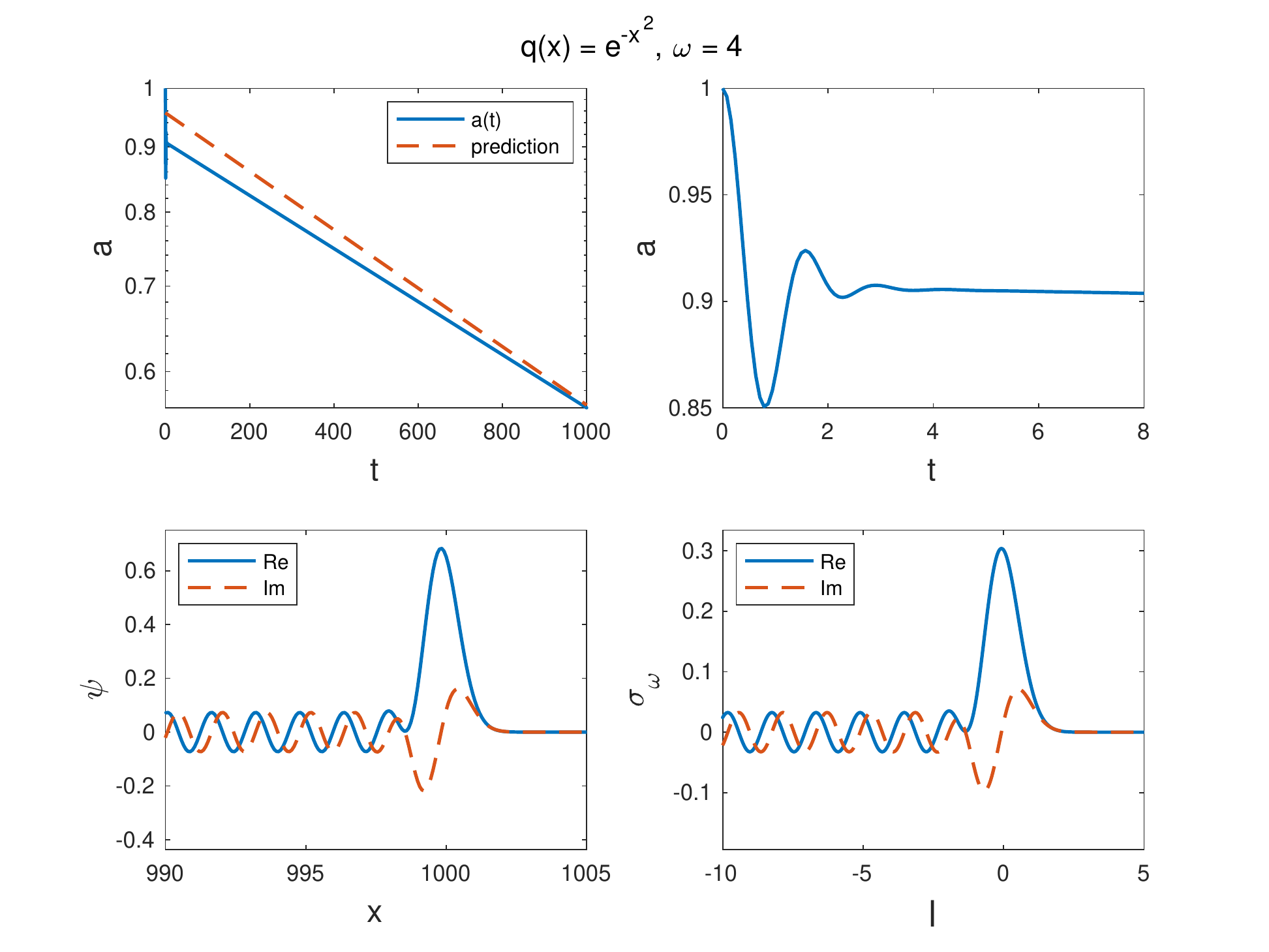}    \caption{}
      \label{gaussianfig4}
\end{figure}
\begin{figure}
\centering
    \includegraphics[width=.75\textwidth]{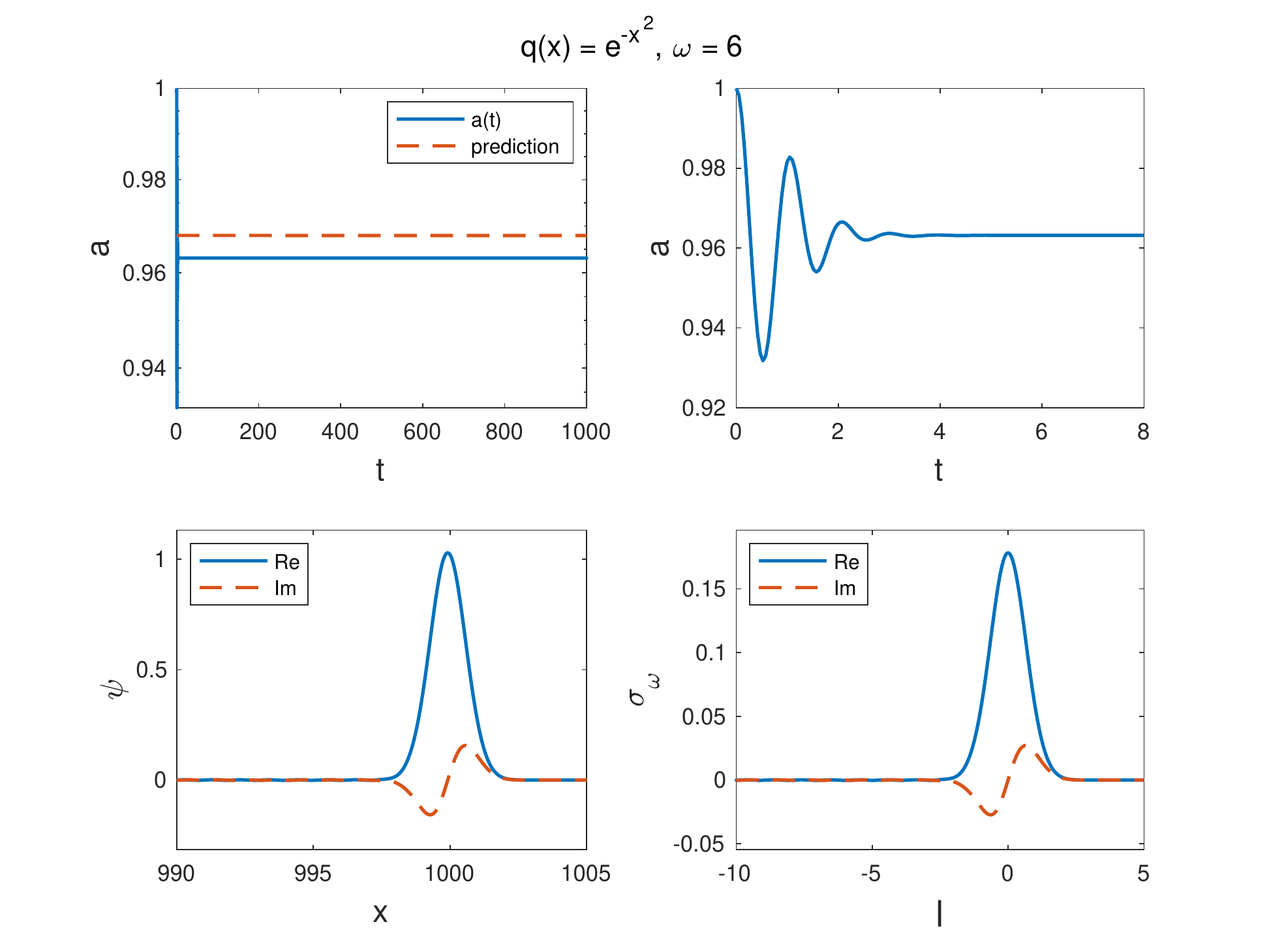}    \caption{}
     \label{gaussianfig6}
\end{figure}

  \begin{figure}
\centering
    \includegraphics[width=.75\textwidth]{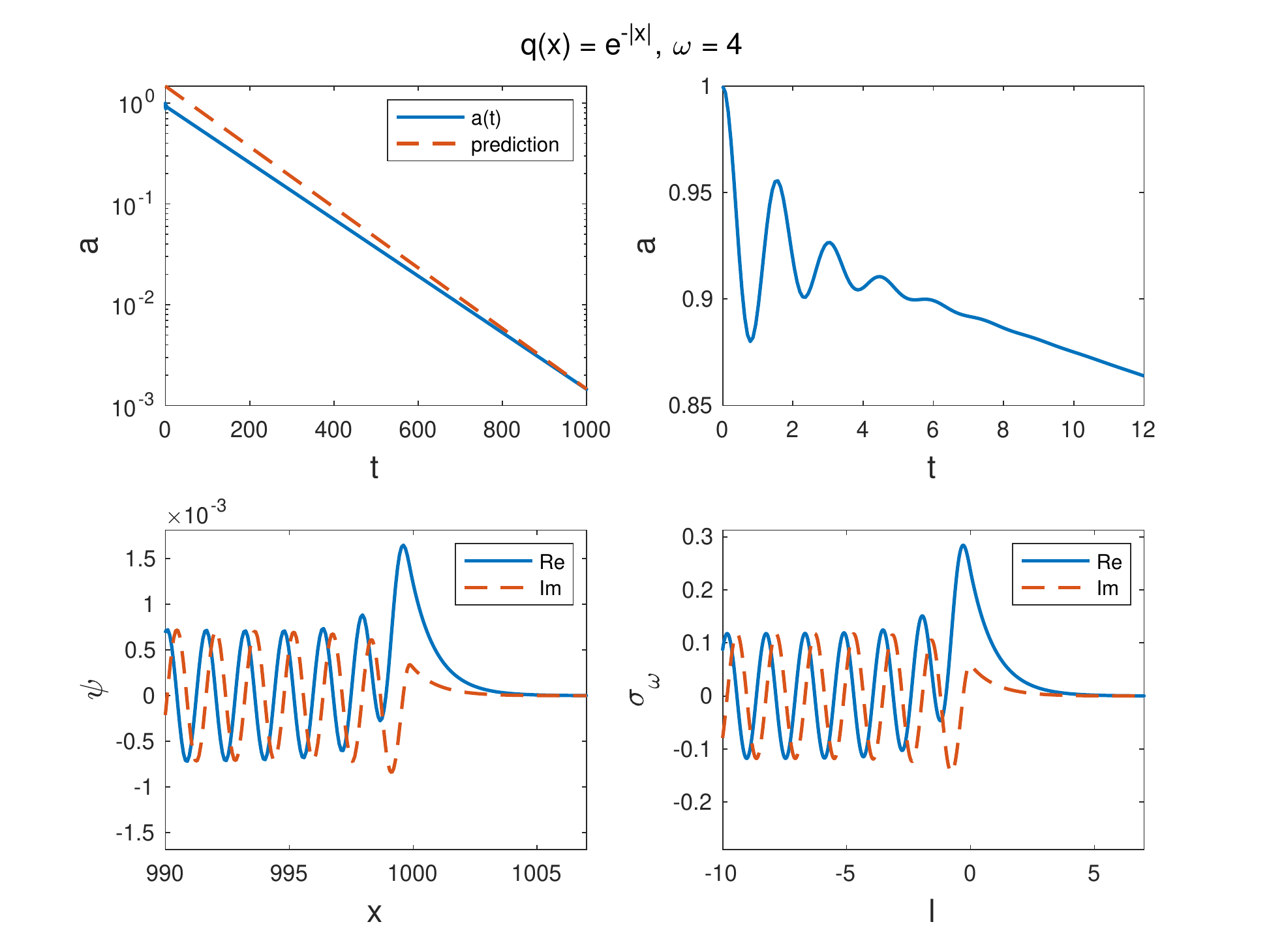}    \caption{}
      \label{peakonfig4}
\end{figure}
\begin{figure}
\centering
    \includegraphics[width=.75\textwidth]{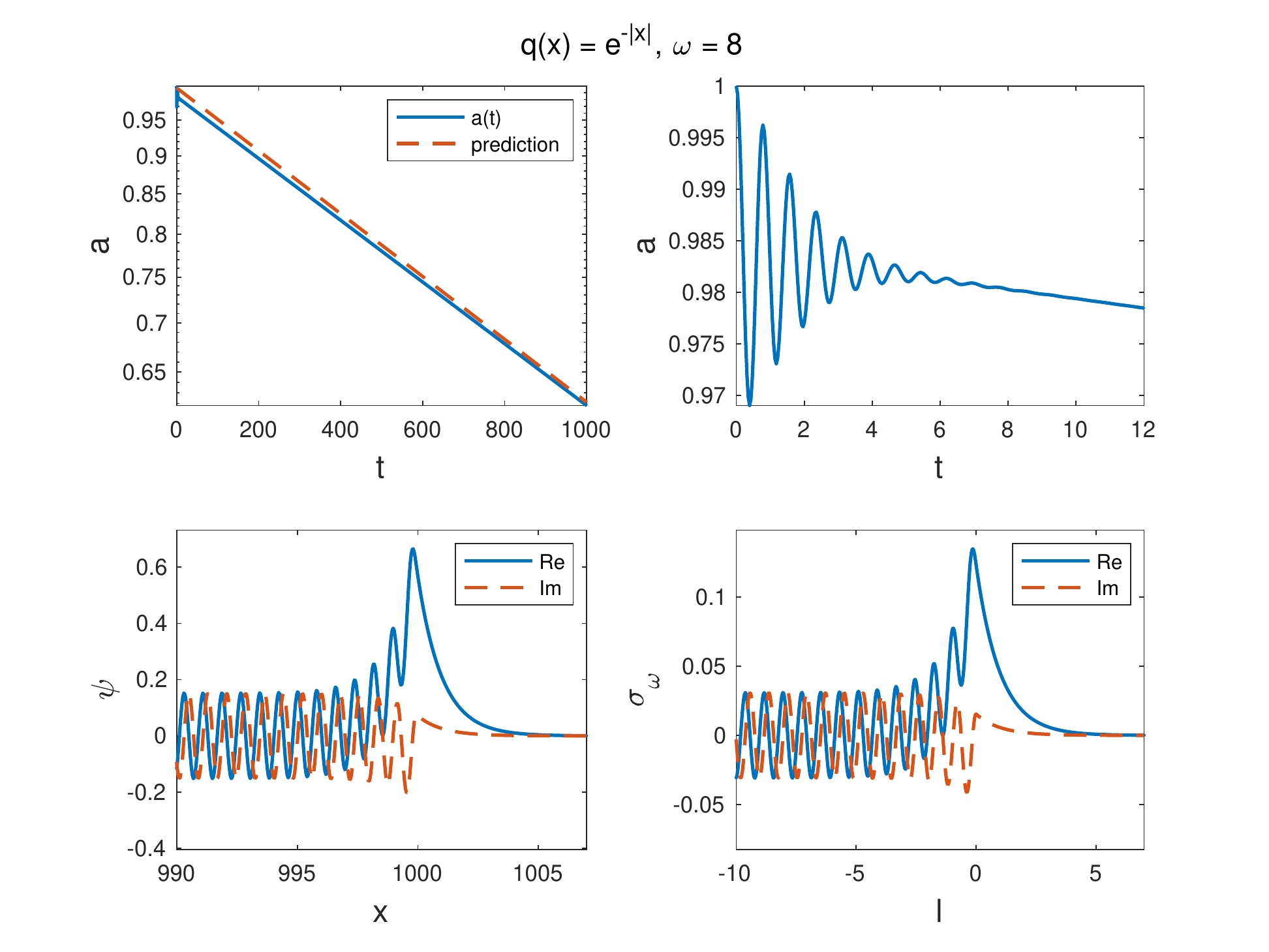}    \caption{}
    \label{peakonfig8}
\end{figure}
\begin{figure}
\centering
    \includegraphics[width=.75\textwidth]{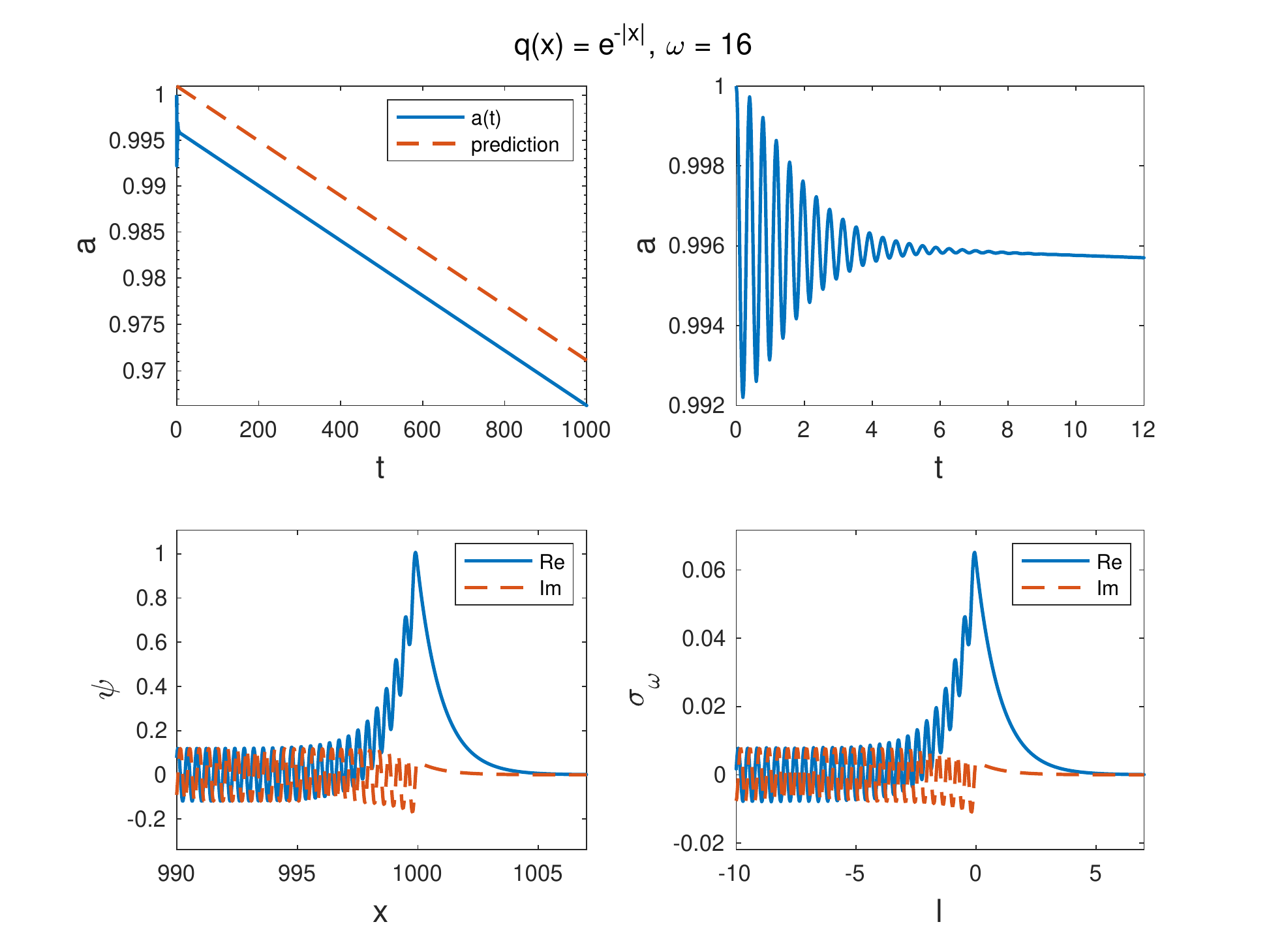}    \caption{}
    \label{peakonfig16}
\end{figure}

\begin{figure}
\centering
    \includegraphics[width=.75\textwidth]{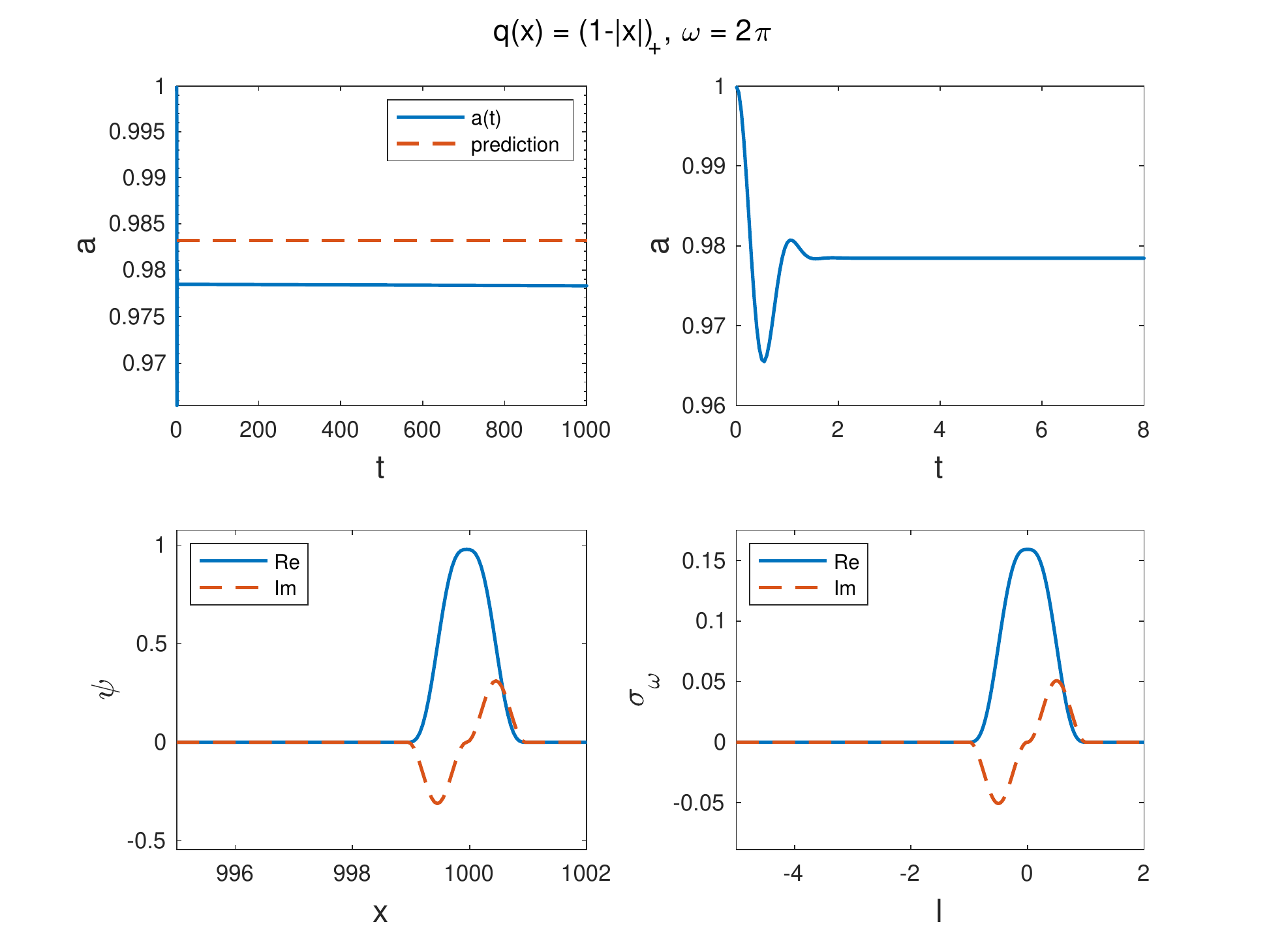}    \caption{}
    \label{tentfig2pi}
\end{figure}
\begin{figure}
\centering
    \includegraphics[width=.75\textwidth]{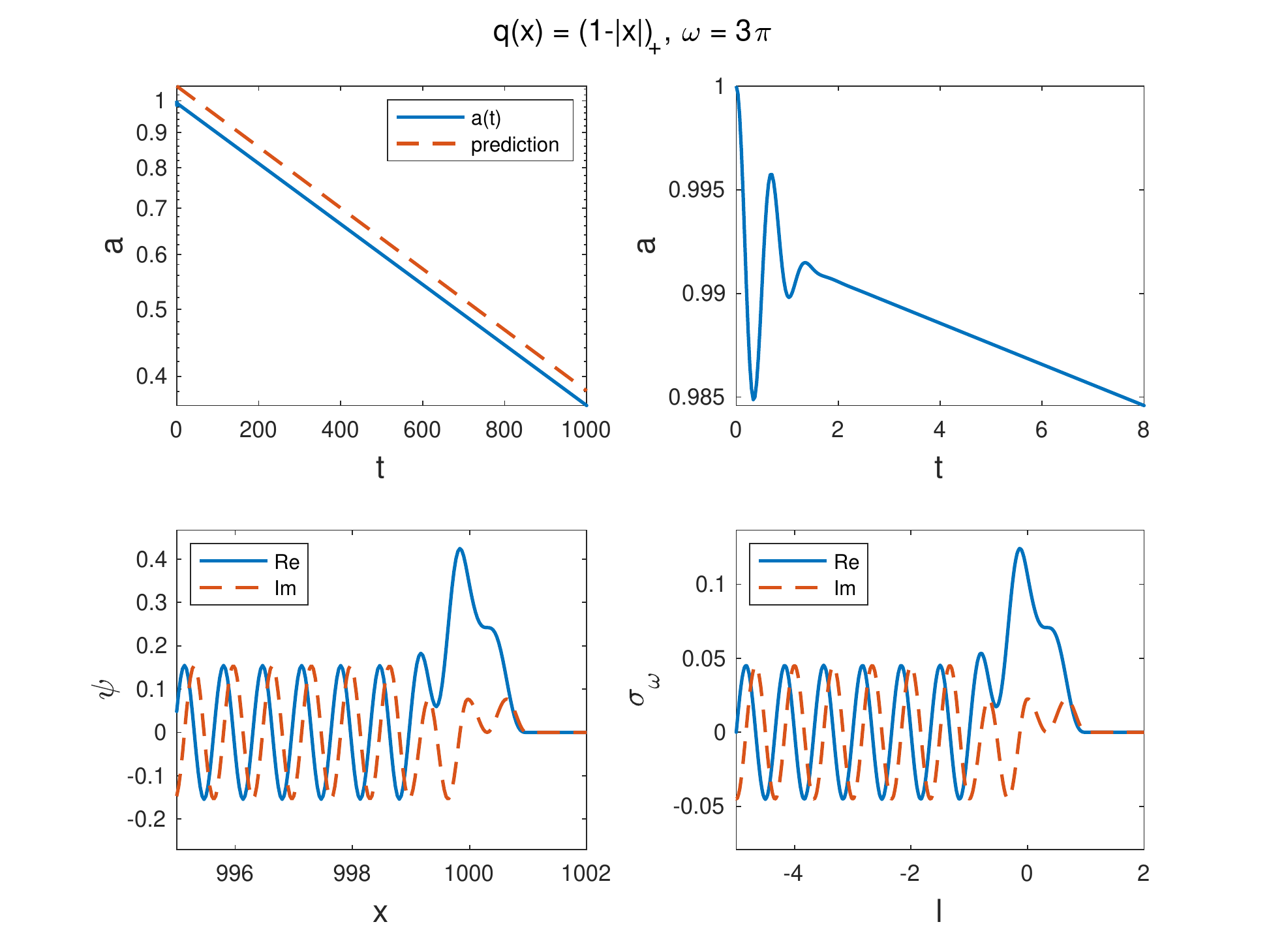}    \caption{}
      \label{tentfig3pi}
\end{figure}
\begin{figure}
\centering
    \includegraphics[width=.75\textwidth]{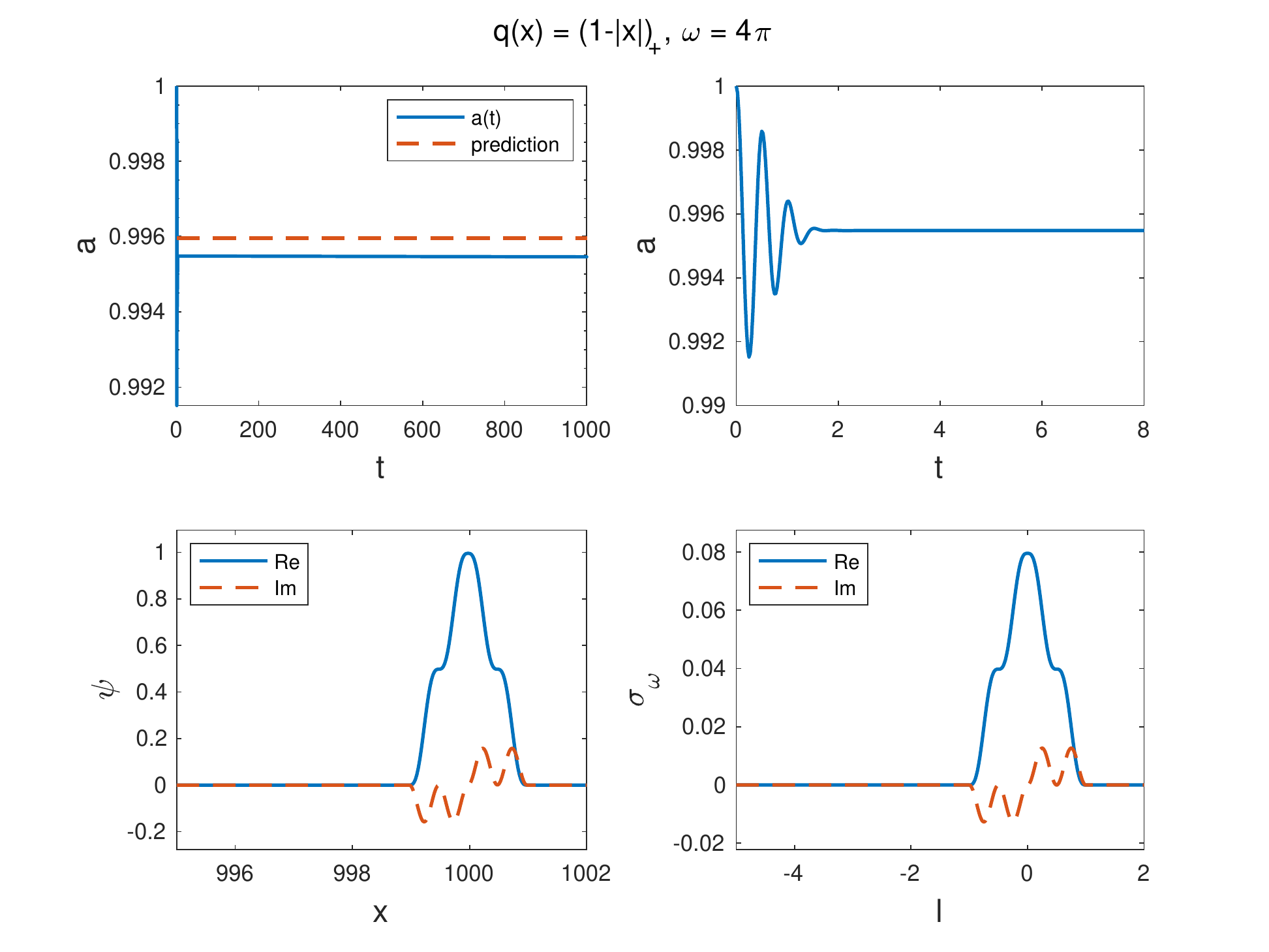}    \caption{}
    \label{tentfig4pi}
\end{figure}
\begin{figure}
\centering
    \includegraphics[width=.75\textwidth]{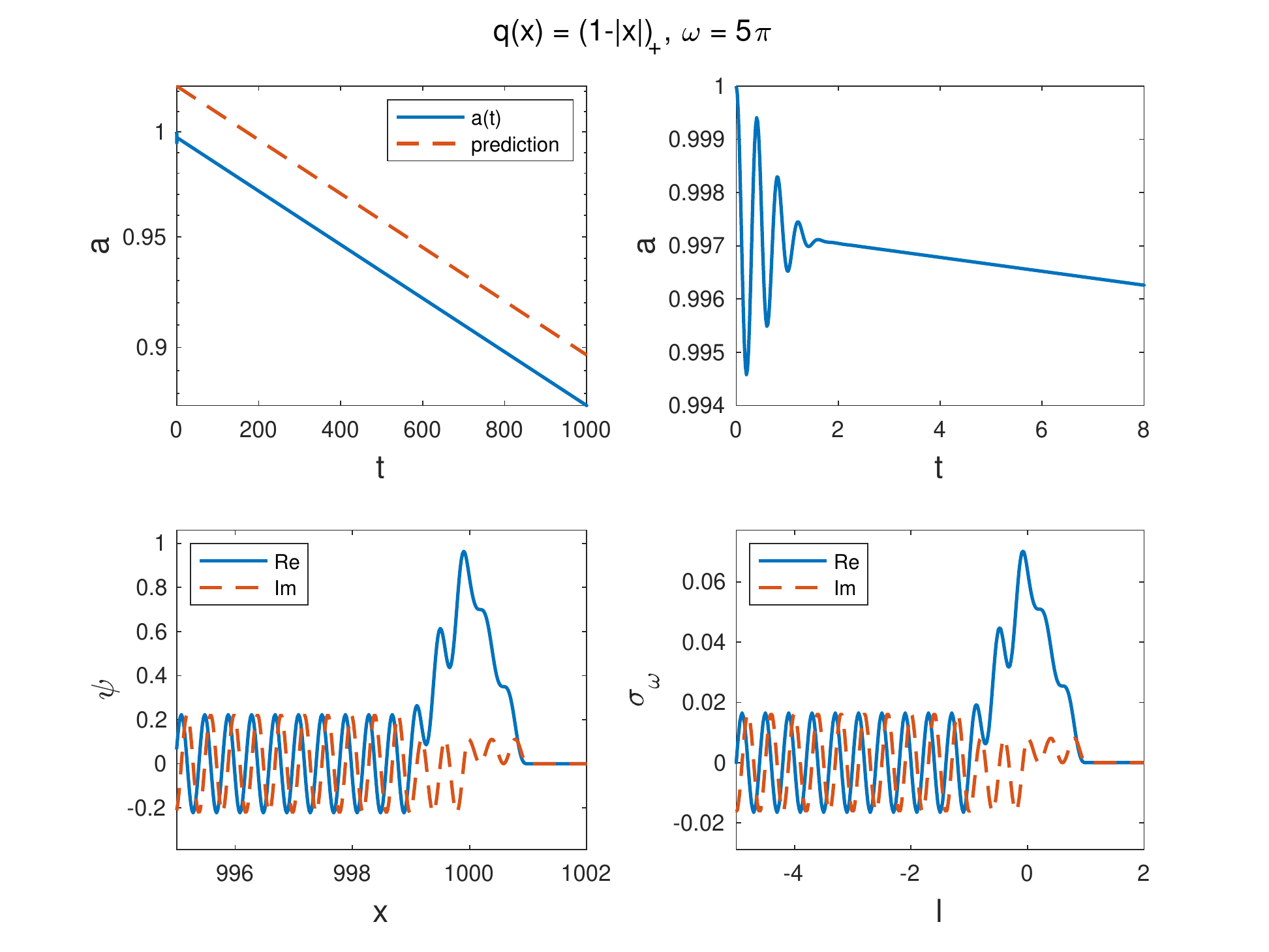}    \caption{}
    \label{tentfig5pi}
\end{figure}

\bibliographystyle{siam}
\bibliography{simplestmodel-bib}{}

\begin{thebibliography}{10}

\bibitem{akylas-yang}
{\sc T.~R. Akylas and T.-S. Yang}, {\em On short-scale oscillatory tails of
  long-wave disturbances}, Stud. Appl. Math., 94 (1995), pp.~1--20.

\bibitem{beale}
{\sc J.~T. Beale}, {\em Exact solitary water waves with capillary ripples at
  infinity}, Comm. Pure Appl. Math., 44 (1991), pp.~211--257.

\bibitem{bellman-cooke}
{\sc R.~Bellman and K.~L. Cooke}, {\em Differential-difference equations},
  Academic Press, New York-London, 1963.

\bibitem{benilov}
{\sc E.~S. Benilov, R.~Grimshaw, and E.~P. Kuznetsova}, {\em The generation of
  radiating waves in a singularly-perturbed {K}orteweg-de {V}ries equation},
  Phys. D, 69 (1993), pp.~270--278.

\bibitem{bona-etal}
{\sc J.~L. Bona, V.~A. Dougalis, and D.~E. Mitsotakis}, {\em Numerical solution
  of {B}oussinesq systems of {K}d{V}-{K}d{V} type. {II}. {E}volution of
  radiating solitary waves}, Nonlinearity, 21 (2008), pp.~2825--2848.

\bibitem{chirilus-bruckner-etal}
{\sc M.~Chirilus-Bruckner, C.~Chong, O.~Prill, and G.~Schneider}, {\em Rigorous
  description of macroscopic wave packets in infinite periodic chains of
  coupled oscillators by modulation equations}, Discrete Contin. Dyn. Syst.
  Ser. S, 5 (2012), pp.~879--901.

\bibitem{faver-spring}
{\sc T.~E. Faver}, {\em Nanopteron-stegoton traveling waves in spring dimer
  {F}ermi-{P}asta-{U}lam-{T}singou lattices}, Quart. Appl. Math., 78 (2020),
  pp.~363--429.

\bibitem{faver-mim}
\leavevmode\vrule height 2pt depth -1.6pt width 23pt, {\em Small mass
  nanopteron traveling waves in mass-in-mass lattices with cubic {FPUT}
  potential}, J. Dynam. Differential Equations, 33 (2021), pp.~1711--1752.

\bibitem{faver-goodman-wright}
{\sc T.~E. Faver, R.~H. Goodman, and J.~D. Wright}, {\em Solitary waves in
  mass-in-mass lattices}, Z. Angew. Math. Phys., 71 (2020), pp.~Paper No. 197,
  20.

\bibitem{faver-wright}
{\sc T.~E. Faver and J.~D. Wright}, {\em Exact diatomic
  {F}ermi-{P}asta-{U}lam-{T}singou solitary waves with optical band ripples at
  infinity}, SIAM J. Math. Anal., 50 (2018), pp.~182--250.

\bibitem{friesecke-pego}
{\sc G.~Friesecke and R.~L. Pego}, {\em Solitary waves on
  {F}ermi-{P}asta-{U}lam lattices. {IV}. {P}roof of stability at low energy},
  Nonlinearity, 17 (2004), pp.~229--251.

\bibitem{GMWZ}
{\sc J.~Gaison, S.~Moskow, J.~D. Wright, and Q.~Zhang}, {\em Approximation of
  polyatomic {FPU} lattices by {K}d{V} equations}, Multiscale Model. Simul., 12
  (2014), pp.~953--995.

\bibitem{nachbin-etal}
{\sc J.~Garnier, J.~C. Mu\~{n}oz Grajales, and A.~Nachbin}, {\em Effective
  behavior of solitary waves over random topography}, Multiscale Model. Simul.,
  6 (2007), pp.~995--1025.

\bibitem{GSWW}
{\sc N.~Giardetti, A.~Shapiro, S.~Windle, and J.~D. Wright}, {\em Metastability
  of solitary waves in diatomic {FPUT} lattices}, Math. Eng., 1 (2019),
  pp.~419--433.

\bibitem{hadadifard-wright}
{\sc F.~Hadadifard and J.~D. Wright}, {\em Mass-in-mass lattices with small
  internal resonators}, Stud. Appl. Math., 146 (2021), pp.~81--98.

\bibitem{HHMCKKYV}
{\sc S.~Hauver, X.~He, D.~Mei, E.~G. Charalampidis, P.~G. Kevrekidis, E.~Kim,
  J.~Yang, and A.~Vainchtein}, {\em Lattices with internal resonator defects},
  Phys. Rev. E, 98 (2018), p.~032902.

\bibitem{hoffman-wright}
{\sc A.~Hoffman and J.~D. Wright}, {\em Nanopteron solutions of diatomic
  {F}ermi-{P}asta-{U}lam-{T}singou lattices with small mass-ratio}, Phys. D,
  358 (2017), pp.~33--59.

\bibitem{holmer}
{\sc J.~Holmer}, {\em Dynamics of {K}d{V} solitons in the presence of a slowly
  varying potential}, Int. Math. Res. Not. IMRN,  (2011), pp.~5367--5397.

\bibitem{johnson-wright}
{\sc M.~A. Johnson and J.~D. Wright}, {\em Generalized solitary waves in the
  gravity-capillary {W}hitham equation}, Stud. Appl. Math., 144 (2020),
  pp.~102--130.

\bibitem{joshi-lustri}
{\sc N.~Joshi and C.~J. Lustri}, {\em Generalized solitary waves in a
  finite-difference {K}orteweg--de {V}ries equation}, Stud. Appl. Math., 142
  (2019), pp.~359--384.

\bibitem{korner}
{\sc T.~W. K\"{o}rner}, {\em Fourier analysis}, Cambridge University Press,
  Cambridge, 1988.

\bibitem{liu-yue}
{\sc Y.~Liu and D.~K.~P. Yue}, {\em On generalized {B}ragg scattering of
  surface waves by bottom ripples}, J. Fluid Mech., 356 (1998), pp.~297--326.

\bibitem{lombardi}
{\sc E.~Lombardi}, {\em Oscillatory integrals and phenomena beyond all
  algebraic orders}, vol.~1741 of Lecture Notes in Mathematics,
  Springer-Verlag, Berlin, 2000.
\newblock With applications to homoclinic orbits in reversible systems.

\bibitem{lustri-porter}
{\sc C.~J. Lustri and M.~A. Porter}, {\em Nanoptera in a period-2 {T}oda
  chain}, SIAM J. Appl. Dyn. Syst., 17 (2018), pp.~1182--1212.

\bibitem{nachbin-papanicolaou}
{\sc A.~Nachbin and G.~C. Papanicolaou}, {\em Water waves in shallow channels
  of rapidly varying depth}, J. Fluid Mech., 241 (1992), pp.~311--332.

\bibitem{nakoulima-etal}
{\sc O.~Nakoulima, N.~Zahibo, E.~Pelinovsky, T.~Talipova, and A.~Kurkin}, {\em
  Solitary wave dynamics in shallow water over periodic topography}, Chaos, 15
  (2005), pp.~037107, 8.

\bibitem{okada}
{\sc Y.~Okada, S.~Watanabe, and H.~Tanaca}, {\em Solitary wave in periodic
  nonlinear lattice}, Journal of the Physical Society of Japan, 59 (1990),
  pp.~2647--2658.

\bibitem{pego-weinstein}
{\sc R.~L. Pego and M.~I. Weinstein}, {\em Asymptotic stability of solitary
  waves}, Comm. Math. Phys., 164 (1994), pp.~305--349.

\bibitem{pelinovsky-schneider}
{\sc D.~E. Pelinovsky and G.~Schneider}, {\em The monoatomic {FPU} system as a
  limit of a diatomic {FPU} system}, Appl. Math. Lett., 107 (2020), pp.~106387,
  8.

\bibitem{schneider-wayne}
{\sc G.~Schneider and C.~E. Wayne}, {\em The rigorous approximation of
  long-wavelength capillary-gravity waves}, Arch. Ration. Mech. Anal., 162
  (2002), pp.~247--285.

\bibitem{sun}
{\sc S.~M. Sun}, {\em Existence of a generalized solitary wave solution for
  water with positive {B}ond number less than {$1/3$}}, J. Math. Anal. Appl.,
  156 (1991), pp.~471--504.

\bibitem{tabata}
{\sc Y.~Tabata}, {\em Stable solitary wave in diatomic toda lattice}, Journal
  of the Physical Society of Japan, 65 (1996), pp.~3689--3691.

\bibitem{tan-etal}
{\sc Y.~Tan, J.~Yang, and D.~E. Pelinovsky}, {\em Semi-stability of embedded
  solitons in the general fifth-order {K}d{V} equation}, Wave Motion, 36
  (2002), pp.~241--255.

\end{thebibliography}

\end{document}